\documentclass[pdflatex,bst/sn-mathphys-ay]{sn-jnl}% Math and Physical Sciences Author Year Reference Style
\usepackage{graphicx}%
\usepackage{multirow}%
\usepackage{amsmath,amssymb,amsfonts}%
\usepackage{amsthm}%
\usepackage{mathrsfs}%
\usepackage[title]{appendix}%
\usepackage{xcolor}%
\usepackage{textcomp}%
\usepackage{manyfoot}%
\usepackage{booktabs}%
\usepackage{algorithm}%
\usepackage{algorithmicx}%
\usepackage{algpseudocode}%
\usepackage{listings}%
\usepackage{float}
\usepackage{caption}

\usepackage{tikz}
\usetikzlibrary{arrows.meta}

\usepackage{commands}

\theoremstyle{thmstyleone}%
\newtheorem{theorem}{Theorem}%  meant for continuous numbers
\newtheorem{lemma}[theorem]{Lemma}

\newtheorem{proposition}[theorem]{Proposition}% 
\newtheorem{corollary}[theorem]{Corollary}

\theoremstyle{thmstyletwo}%
\newtheorem{example}{Example}%
\newtheorem{remark}{Remark}%

\theoremstyle{thmstylethree}%
\newtheorem{definition}{Definition}%

\begin{document}

\title[Article Title]{Transformation semigroup perspective on the magma monoid}


\author*[1]{\fnm{Yakiv} \sur{Baiduk}}\email{ybaiduk@kse.org.ua}

\author[2]{\fnm{Sergiy} \sur{Kozerenko}}\email{s.kozerenko@kse.org.ua}
 
\affil[1]{\orgdiv{Department of Mathematics}, \orgname{Kyiv School of Economics}, \orgaddress{Mykoly Shpaka str. 3, 03113 \city{Kyiv}, \country{Ukraine}}}

\affil[2]{\orgdiv{Computer Science Department}, \orgname{Kyiv School of Economics}, \orgaddress{Mykoly Shpaka str. 3, 03113 \city{Kyiv}, \country{Ukraine}}}

\abstract{The monoid of all binary operations was first introduced by H.~S.~Kim and J.~Neggers in 2008. Since then, different aspects and applications of this monoid were studied, while several questions about its semigroup-theoretic properties remain unanswered. We employ a transformation semigroup perspective to fully characterize principal left and right ideals, idempotent and regular elements of this monoid, as well as provide precise combinatorial enumerations of them. This approach gives a general framework for most of the existing results on ideals in the magma monoid. We also answer several open questions posed in the 2023 PhD dissertation of A.~Rafieipour. Finally, we correct an error regarding the description of the center of the magma monoid from the 2011 paper of H.~F.~Fayomi.}



\maketitle

\newpage

\section{Introduction and preliminaries}\label{sec1}
Throughout this article, $X$ is a finite set of cardinality $n$. \\

Denote by $\M(X)$ (or simply by $\M$) the set of all possible binary operations on $X$:
\[
    \M(X) = X^{X \times X}.
\]

In 2008, H.~S.~Kim and J.~Neggers~\cite{BinarySystems2008} introduced the binary operation $\trg$ on $\M$ (although the notation at that time was different).

\begin{definition}
    Let $\trg: \M \times \M \to \M$ be the binary operation defined as
    \[
    a (\circ \trg \ast) b = (a \circ b) \ast (b \circ a)
    \] for all  $\circ, \ast \in \M$ and $a, b \in X$.
\end{definition}

\begin{proposition}{\rm\cite[Theorem 2]{BinarySystems2008}}\label{prop-1}
 For any set $X$, the pair $(\M(X), \trg)$ is a monoid with the identity element being the left zeros operation $\circ_{lz}$ (defined as $a \circ_{lz} b = a$ for all $a, b \in X$).   
\end{proposition}

In light of Proposition~\ref{prop-1}, $\M(X)$ is also known as the \textit{magma monoid} on $X$. Various interesting structures on $\M$ were introduced over the past years, and some connections with other fields were built. For example, a family of submonoids of $\M$ given by outsets $Out(\ast) = \{\circ \in \M: \ast \text{ distributes over } \circ\}$ was analyzed in~\cite{Lopez2022}. Moreover, the kernel-cokernel decomposition of binary operations was introduced in \cite[Definition~4.2]{Lopez2022}, and the group of units of $\M$ was fully described \cite[Theorem~3.5]{Lopez2022}. On top of that, two families of graph-induced binary operations were studied in~\cite{Rafieipour2023}, showing that these families are subsemigroups of $\M$, and retrieving some of their structural properties. For the information on the global structure of distributivity relations on $\M$ we refer to the work~\cite{Distributivity2021}. \\

However, to the best of our knowledge, a wide range of algebraic properties of $\M$ remains unknown to this day. A little is known about ideals and Green`s relations, idempotent elements, and regularity of operations. \textit{Some} families of idempotents were described, \textit{some} sets were shown to be ideals in $\M$, but a lack of general characterizations of key semigroup-theoretic properties is clear. In fact, problems regarding these core properties are explicitly stated as open in \cite{Rafieipour2023}. \\

This work aims to close the mentioned gap by providing a clean and general description of principal ideals and Green's relations, idempotents, and regular elements. We also introduce a new decomposition of binary operations, which helps us to extract precise combinatorial results on key objects, such as the total number of idempotent elements in $\M$ or cardinalities of Green's classes.

Most of our results come out naturally after building a right framework. We embed $\M(X)$ into the transformation semigroup $\T_{X \times X}$ of all self-mappings of $X \times X$. Transformation semigroups are very well studied~\cite{GanyushkinMazor2009}, and most of their structural properties are transferred to $\M$ almost without any changes, although some preliminary work is needed to make this transition seamless. \\

The major part of results from this article was announced at the International Conference dedicated to the 75th Anniversary of the birth of Volodymyr Masliuchenko~\cite{BaidukChernivtsi2025} and Ukraine Mathematics Conference ``At the End of the Year 2025''~\cite{BaidukKozerenko2025IdempotentMagma}.

\section{The pairmorph transformation}
When attempting to retrieve some information about the algebraic structure of $\M$, one sooner or later faces the problem that for an arbitrary $\ast \in \M$, the elements $a \ast b$ and $b \ast a$ are completely unrelated. This significantly complicates the proofs and clearly indicates the need for a different approach. A natural idea is therefore to consider both $a \ast b$ and $b \ast a$ simultaneously for each pair $(a,b) \in X \times X$.

\begin{definition}
    Let $\ast \in \M$ be a binary operation. \textit{The pairmorph transformation} induced by $\ast$ is a transformation $T_\ast: X \times X \to X \times X$ defined as
    \[
    T_\ast(a, b) := (a \ast b, b \ast a).
    \]
    for all $(a,b)\in X\times X$.
\end{definition}

This idea is not new and has already appeared in \cite[Theorem~3.5]{Lopez2022}, where it was used to fully characterize the group of units in $\mathcal{M}$. In that setting, however, it arose as a specific instance of what is referred to as $\ast \times \circ\zeta$. What is new to our approach is the fact that we consider pairmorph transformations as the main perspective, rarely working with binary operations directly. \\

The main reason for studying pairmorph transformations is the fact that they induce a monomorphism of $\M$ into $\T_{X \times X}$. In other words, the magma monoid can be viewed as a submonoid of the full transformation semigroup on $X \times X$. Lemma~\ref{T_ast_criterion} and Theorem~\ref{M_iso_to_T_XxX} establish this monomorphism and fully characterize its image.

\begin{lemma}\label{T_ast_criterion}
    A given transformation $T$ on $X \times X$ is the $T_\ast$ for some $\ast \in \M$ if and only if 
    \[
        T((a, b)^\leftarrow) = (T(a,b))^\leftarrow,
    \]for all $a, b \in X$, where $(u, v)^\leftarrow = (v, u)$.
\end{lemma}
\begin{proof}
   We first prove that this condition is necessary.
    Let $\ast \in \M$ be a binary operation. For any $a, b \in X$, we have 
    \[
        T_\ast((a, b)^\leftarrow) = T_\ast(b, a) = (b\ast a, a \ast b) = (a \ast b, b \ast a)^\leftarrow = (T_\ast(a, b))^\leftarrow.
    \]
    
    Now we prove the sufficiency of this condition. Let $T \in \T_{X \times X}$ be an arbitrary transformation of $X \times X$ that commutes with $^\leftarrow$. We are going to construct an operation $\ast \in \M$ that will have $T_\ast = T$.
    For every $(a, b) \in X \times X$, define $a \ast b = \circ_{lz}(T(a, b))$.

    If $T_\ast(a, b)$ is equal to some $(x, y) \in X \times X$, then we know that $\circ_{lz}(T(a, b)) = x$ and $\circ_{lz}(T(b, a)) = y$. But since $T$ commutes with $^\leftarrow$, $\circ_{lz}(T(b, a))$ is the same as $\circ_{rz}(T(a, b))$, where $\circ_{rz}$ is the right-zeros operation.

    The equation $\circ_{lz}(T(b, a)) = \circ_{rz}(T(a, b))$ allows us to conclude that
    \[
        T_\ast(a, b) = (\circ_{lz}(T(a, b)), \circ_{lz}(T(b, a))) = (\circ_{lz}(T(a, b)), \circ_{rz}(T(a, b))) = T(a, b),
    \] completing the proof.
\end{proof}

\begin{theorem}\label{M_iso_to_T_XxX}
The magma monoid $(\M, \trg)$ is isomorphic to a submonoid of all transformations $T \in \T_{X \times X}$ that commute with $^\leftarrow$.
\end{theorem}
\begin{proof}
    Denote by $\T^\leftrightarrow$ the set of all transformations of $X \times X$ for which $T((a, b)^\leftarrow) = (T(a, b))^\leftarrow$. Let $\phi: \M \to \T^\leftrightarrow$ be a map such that $\phi(\ast) = T_\ast$. We are going to prove that $\phi$ is an isomorphism of $\M$ to $\T^\leftrightarrow$. Clearly, $\phi$ is injective, and its surjectivity is a corollary of Lemma~\ref{T_ast_criterion}. Thus, $\phi$ is a bijection. The only thing left is to show that it is also a homomorphism.

    Notice that for arbitrary operations $\ast, \circ \in \M$ and any $a, b \in X$ the product operation $a (\circ \trg \ast) b$ can be expressed as $\ast(T_\circ(a, b))$. Because of that,
    \begin{align*}
        \phi(\circ \trg \ast)(a, b) &= T_{\circ \trg \ast}(a,b) = (a (\circ \trg \ast) b, b(\circ \trg \ast)a) = (\ast(T_\circ(a, b)), \ast(T_\circ(b,a))) \\
        &= (\ast(T_\circ(a,b)), \ast((T_\circ(a,b))^\leftarrow)) = T_\ast(T_\circ(a, b)).
    \end{align*}

    Therefore, $\phi$ is an isomorphism of $\M$ to $\T^\leftrightarrow.$
\end{proof}

With Theorem~\ref{M_iso_to_T_XxX} in mind, we now use binary operations and  their pairmorph transformations interchangeably, because their semigroup-theoretic properties are identical. This allows us to transfer knowledge about $\mathcal{T}_{X \times X}$ to $\mathcal{M}$. We begin by adapting some existing notation.

 \begin{definition}
     For a binary operation $\ast \in \M$, define its \textit{pairmorph image} as $im(T_\ast) = T_\ast(X \times X)$ and its \textit{diagonal image} as $diag(T_\ast) = T_\ast(\Delta(X))$.
 \end{definition}

The equality $T_\ast(b, a) = (T_\ast(a, b))^\leftarrow$ is a strong structural restriction that naturally brings us to the concept of a tournament. A \textit{tournament} on a set $X$ is an orientation of a complete graph on $X$ without loops. In other words, for every two-element subset $\{a, b\} \subset X$, a tornament contains either $(a, b)$ or $(b, a)$, and never both. We will denote by $\omega$ an arbitrary tournament on $X$. Clearly, $T_\ast$ is fully determined by its values only on $\omega$ and the diagonal $\Delta(X) = \{(x, x): x \in X\}$. This will be expanded upon in Section~\ref{decompositions}.

\section{Principal left ideals of $\M$ and decompositions of pairmorph transformations}\label{decompositions}
Recall that in a monoid $S$, the principal left ideal generated by an element $a \in S$ is the set $Sa = \{xa: x \in S\}$. The left Green's relation $\L$ is defined as $a \L b \iff Sa = Sb$. \\

Similarly to transformation semigroups, where the image of a transformation completely determines its principal left ideal, the pairmorph and the diagonal images describe principal left ideals in $\M$.
\begin{theorem}\label{L_princ_ideals}
    For any $\ast \in \M$, the principal left ideal generated by it is precisely the set
    \[
        \M\ast = \{\dmd \in \M: im(T_\dmd) \subset im(T_\ast) \text{ and } diag(T_\dmd) \subset diag(T_\ast)\}.
    \]
\end{theorem}

\begin{proof}
    We begin by proving the inclusion from left to right. Let $\dmd$ be an element of the principal left ideal of $\ast \in \M$. This means that there is some $\circ \in \M$ with $\dmd = \circ \trg \ast$. Let $(x, y)$ be an element of $im(T_\dmd)$. Then there is some pair $(a, b) \in X \times X$ such that $T_\dmd(a, b) = (x, y)$. Notice that 
    \[
    T_\dmd(a, b) = T_{\circ \trg \ast}(a, b) = T_\ast(T_\circ(a, b)) = (x, y).
    \] Thus, $(x, y)$ also belongs to $im(T_\ast)$. This reasoning is identical for any element $(y, y) \in diag(T_\dmd)$.\\

     Going from right to left, let $\dmd \in \M$ be such that $im(T_\dmd)$ is a subset of $im(T_\ast)$ and $diag(T_\dmd)$ is a subset of $diag(T_\ast)$. We are going to construct a $\circ \in \M$ such that $\dmd = \circ \trg \ast$, showing that $\dmd$ is in the principal left ideal of~$\ast$. We will construct the $\circ$ by defining its $T_\circ$ and utilizing Lemma~\ref{T_ast_criterion}. Since $im(T_\dmd) \subset im(T_\ast)$, for any $(a, b) \in X \times X$ there is at least one pair $(x_{a}, x_{b})$ with $T_\dmd(a, b) = T_\ast(x_a, x_b)$. \\

    Let $\omega$ be the arc set of some tournament on $X$. Set a mapping $\phi: \omega \to X \times X$ such that $\phi(a, b) =(x_a, x_b)$ as discussed above. For every $(a, b) \in \omega$, define $T(a, b) = \phi(a, b)$ and $T(b, a) = (\phi(a, b))^\leftarrow$. \\

    Next, we utilize the fact that $diag(T_\dmd) \subset diag(T_\ast)$ to get that for any $a \in X$, there is an element $x_a \in X$ such that $a \dmd a = x_a \ast x_a$. Using this, define a mapping $\psi: \Delta(X) \to \Delta(X)$ as $\psi(a, a) = (x_a, x_a)$. Also, for any $(a, a) \in \Delta(X)$, set $T(a, a)$ to be equal to $\psi(a, a)$. \\

    Note that by construction $T$ satisfies the conditions of Lemma~\ref{T_ast_criterion}, so there exists $\circ \in \M$ such that $T = T_\circ$. The only thing left is to show that $T_{\circ \trg \ast} = T_\dmd$, which will imply that $\dmd = \circ \trg \ast$.

    For any $(a, b) \in X \times X$, we know the following:
    \begin{itemize}

        \item If $(a, b) \in \omega$, then $T_{\circ \trg \ast}(a,b) = T_\ast(\phi(a, b)) = T_\ast(x_a, x_b) = T_\dmd(a, b);$
        \item If $(a, b) \notin \omega$ and $a \neq b$, then $T_{\circ \trg \ast}(a, b) = T_\ast((\phi(a, b))^\leftarrow) = (T_\ast(\phi(b, a)))^\leftarrow = (T_\dmd(b, a))^\leftarrow = T_\dmd(a, b)$;
        \item If $a = b$, then $T_{\circ \trg \ast}(a, b) = T_\ast(\psi (a, b)) = T_\dmd(a, b)$.
    \end{itemize}

    In all cases, $T_{\circ \trg \ast}(a, b) = T_\dmd(a, b)$, from which we conclude that $\dmd = \circ \trg \ast$. This means that $\dmd$ belongs to the principal left ideal of $\ast$.
\end{proof}

\begin{corollary}\label{L_class_char}
    Two operations $\ast, \star \in \M$ are $\L$--equivalent if and only if $im(T_\ast) = im(T_\star)$ and $diag(T_\ast) = diag(T_\star)$.
\end{corollary}

The next logical step is to count the number of $\L$--classes and determine their cardinalities. This problem presents a need for a decomposition of pairmorph transformations into simpler, easily enumerable, and independent pieces. In the previous section, we mentioned the fact that a given pairmorph transformation $T_\ast$ is fully determined by its values on some tournament $\omega$ and on $\Delta(X)$, and the proof of Theorem~\ref{L_princ_ideals} utilizes this to build the corresponding $\circ$. This suggests a way to decompose a given binary operation.

\begin{proposition}\label{phi_psi_decomp}
    There is a bijection between $\M$ and all pairs of mappings $\phi: \Delta(X) \to \Delta(X)$ and $\psi: \omega \to X \times X$.
\end{proposition}

\begin{proof}
    Given $\ast \in \M$, map it to $\phi$ defined as the restriction of $T_\ast$ to $\Delta(X)$ and $\psi$ defined as the restriction of  $T_\ast$ to $\omega$. Given a pair $(\phi, \psi)$, map it to $T_\ast$ such that $T_\ast$ restricted to $\Delta(X)$ is $\phi$, $T_\ast$ restricted to $\omega$ is $\psi$, and for every element $(b, a) \notin \omega$ set $T_\ast(b, a) = (\psi(a, b))^\leftarrow = (T_\ast(a, b))^\leftarrow$.
\end{proof}

In fact, this decomposition is exactly the kernel-cokernel decomposition introduced in~\cite[Definition~4.2]{Lopez2022}. In that work, $ker(\ast)$ is a binary operation, the pairmorph transformation of which acts as identity on $(X\times X) \setminus \Delta(X)$ and as $\phi$ on $\Delta(X)$; and $cok(\ast)$ is a binary operation, the pairmorph transformation of which acts as identity on $\Delta(X)$ and as defined in the proof of Proposition~\ref{phi_psi_decomp} on $(X \times X) \setminus \Delta(X)$. \\

With intention to keep notation consistent, we denote the mappings $\phi$ and $\psi$ from Proposition~\ref{phi_psi_decomp} as $ker_\ast$ and $cok_\ast$, respectively. We place the binary operation as a bottom index to indicate that these objects are now transformations, which allows us to use the notation $cok_\ast(a, b)$.\\

It is clear that the pairmorph image of $\ast$ is equal to the union of images of its kernel and cokernel. This observation leads to the following combinatorial result.

\begin{theorem}\label{num_L_classes}
    There are $\displaystyle \sum_{d=1}^n \binom{n}{d}\sum_{k=0}^{N}\binom{N + n - d}{k}$ different principal left ideals in $\M$, where $N = |\omega| = \binom{n}{2}$.
\end{theorem}
\begin{proof}
    By Corollary~\ref{L_class_char}, $\L$--classes of $\M$ are determined by the pairmorph image and the diagonal image of its members. Therefore, our goal is to enumerate such unique pairs. 
    
    Assume $D \subset \Delta(X)$ is a fixed diagonal image of some binary operation $\ast \in \M$. This means that $im(ker_\ast) = D$. Hence, the rest of the pairmorph image is determined by $im(cok_\ast)$. Recall that $cok_\ast : \omega \to X\times X$. Thus, we are free to choose from $0$ up to $|\omega| =N$ elements not in $D$ that will enlarge the image. Note that there are only $N + n - |D|$ such elements, because when we map an element from $\omega$ to $(x, y)$, its converse must be mapped to $(y, x)$, hence the number of impactful choices decreases. With all that in mind, the number of such images is
    \[
    \displaystyle \sum_{k=0}^{N} \binom{N + n - |D|}{k}.
    \] The claim follows from applying this logic to every possible nonempty $D \subset \Delta(X)$.
\end{proof}

The proof of Theorem~\ref{num_L_classes} demonstrates a very important fact. When considering the impact of some choice of mapping for $cok_\ast$ on the resulting pairmorph image, some choices are dual. This becomes very cumbersome to work with when attempting to compute the cardinalities of $\L$--classes. Moreover, this behavior is not preserved for pairs at which $\ast$ commutes, further complicating the proofs. 

We address these issues by introducing a new decomposition of the cokernel of a given binary operation into the so-called commutative and anticommutative parts. Define the \textit{anticommutative domain} of a binary operation $\ast \in \M$ as 
\[acdom(T_\ast) = \{(a, b) \in X \times X: T_\ast(a, b) \notin \Delta(X)\}.\]

\begin{definition}\label{comm-acomm-decomp}
    Let $\ast \in \M$ be a binary operation. The \textit{commutative-anticommutative decomposition of the cokernel} of $\ast$ is an ordered triplet of mappings $(ccok_\ast, accok_\ast, ori_\ast)$, where
    \begin{itemize}
        \item $ccok_\ast: \omega \setminus acdom(T_\ast) \to \Delta(X)$ is given by $ccok_\ast(a, b) = T_\ast(a, b)$;
        \item $accok_\ast: acdom(T_\ast) \cap \omega \to \omega$ is given by $accok_\ast(a, b) = \begin{cases}
            T_\ast(a, b), \text{ if } T_\ast(a, b) \in \omega, \\
            T_\ast(b, a), \text{ otherwise};
        \end{cases}$
        \item $ori_\ast: acdom(T_\ast) \cap \omega \to \{0, 1\}$ is given by $ori_\ast(a, b) = \begin{cases}
            0, \text{if } T_\ast(a, b) \in \omega, \\
            1, \text{otherwise}.
        \end{cases}$
    \end{itemize}
\end{definition}

In other words, $accok_\ast$ ``forgets'' the order of output pairs outside of $\Delta(X)$, and $ori_\ast$ allows to reconstruct the orientation back. Therefore, there is a natural bijection between the set of all cokernels of all possible binary operations with a fixed $acdom(T_\ast)$, and all such triplets of mappings. 

\begin{example}\label{decomp_example}
    Let $X  = \{0, 1, 2\}$ and $\omega = \{(0, 1), (0, 2), (1, 2)\}$. Consider a binary operation $\ast \in \M$ given by the Cayley table presented in Table~\ref{tab:cayley1}. Then every part of the decomposition can be seen on the Figure~\ref{figure-1}: dashed arrows representing $accok_\ast$, dotted arrows representing $ccok_\ast$, solid black arrows representing $ker_\ast$, and colored vertices having the value of $ori_\ast$ equal to $1$ at them. The visualisation on Figure~\ref{figure-1} is reffered to as the \textit{pairmorph graph} of a binary operation $\ast \in \M$.

\begin{figure}[ht]
\centering

\begin{minipage}{0.6\textwidth}
\centering
\begin{tikzpicture}[node distance=16mm, scale=0.45,
  main/.style = {draw, circle, minimum size=0.6cm}]
    \node[main] (1) at (0,0) {$(1,1)$};
    \node[main] (2) at (0,4) {$(2,2)$};
    \node[main] (3) at (0,8) {$(0,0)$};
    \node[main] (4) at (5,4) {$(0,1)$};
    \node[main] (5) at (5,8) {$(0,2)$};
    \node[main ,fill=gray] (6) at (5,0) {$(1,2)$};

    \draw[-{Latex[length=1.2mm,width=2.4mm]}]
      (2) edge (1)
      (3) edge (2)
      (4) edge[blue,dotted] (2)
      (5) edge[red,dashed] (4)
      (6) edge[red,dashed] (4)
      (1) edge[in=-110, out=-70, looseness=5] (1);
\end{tikzpicture}

\caption{The illustration of $ker_\ast$ (in solid black), $ccok_\ast$ (dotted), $accok_\ast$ (dashed) and $ori_\ast$ (colored vertex) for $\ast$ from Example~\ref{decomp_example}.}
\label{figure-1}
\end{minipage}
\hfill
\begin{minipage}{0.38\textwidth}
\centering
\begin{tabular}{c|c|c|c}
$\ast$ & $0$ & $1$ & $2$ \\
\hline
$0$ & $2$ & $2$ & $0$ \\
\hline
$1$ & $2$ & $1$ & $1$ \\
\hline
$2$ & $1$ & $0$ & $1$
\end{tabular}

\captionof{table}{The Cayley table for $\ast$ from Example~\ref{decomp_example}.}
\label{tab:cayley1}
\end{minipage}

\end{figure}

For example, to determine the value of $T_\ast(1, 2)$ using the presented graph, one first looks at the head of the arrow pointing from $(1, 2)$, which is $(0, 1)$. Next, because $(1, 2)$ is colored, $ori_\ast(1, 2) = 1$, hence $T_\ast(1, 2) = (0, 1)^\leftarrow = (1, 0)$.

\end{example}

The commutative-anticommutative decomposition of the cokernel makes the process of counting binary operations with a fixed pairmorph image much easier. We simply count all options for $ker_\ast, ccok_\ast, accok_\ast$, and then multiply the result by $2^{|acdom(T_\ast)|}$ to account for all possible $ori_\ast$. \\

As preparation to counting cardinalities of left Green's classes, we build some additional notation. For a binary operation $\ast\in\M$, define the 
\begin{itemize}
    \item \textit{commutative image} of $T_\ast$ as $cim(T_\ast) = im(T_\ast) \cap \Delta(X)$;
    \item \textit{anticommutative image} of $T_\ast$ as $acim(T_\ast) = im(accok_\ast)$;
    \item \textit{diagonal rank} of $T_\ast$ as $drank(T_\ast) = |diag(T_\ast)|$;
    \item \textit{commutative rank} of $T_\ast$ as $crank(T_\ast) = |cim(T_\ast)|$;
    \item \textit{anticommutative rank} of $T_\ast$ as $arank(T_\ast) = |acim(T_\ast)|$;
    \item \textit{commutative defect} of $T_\ast$ as $cdef(T_\ast) = |cim(T_\ast) \setminus diag(T_\ast)| = crank(T_\ast)-drank(T_\ast)$.
\end{itemize}

\begin{theorem}\label{L_class_cardinality}
    Let $\ast \in \M$ be a binary operation with $drank(T_\ast) =d$, $crank(T_\ast) = c$, $arank(T_\ast) = a$. Then the cardinality of its left Green's class is exactly
    \[
        \Surj(n, d)\sum_{r = a}^{N-c + d}\binom{N}{r}\Surj(r, a) 2^r\left( c^{N-r} + \sum_{k = 1}^{c - d}\left((-1)^k\binom{c-d}{k}(c-k)^{N-r}\right)\right),
    \]
    where $N = |\omega| =  \binom{n}{2}$; $\Surj(m, n) = \surj{m}{n}{i}$ is the number of all surjective mappings from an $m$--element set onto an $n$--element set for $n+m>0$, and $\Surj(0, 0):=1$.
\end{theorem}
\begin{proof}
    Let $\ast \in \M$ be a binary operation with $diag(T_\ast) = D$, $cim(T_\ast) = C$, $acim(T_\ast) = A$. Then $|D| = d$, $|C| = c$, $|A| = a$. We count how many binary operations also have the same diagonal, commutative and anticommutative images. For this, we count all options for the $ker_\ast$, $ccok_\ast$, $accok_\ast$, and $ori_\ast$. The kernel of $\ast$ must be a surjective mapping of $\Delta(X)$ onto $D$, hence there are $\Surj(n, d)$ possible kernels. \\
    
    We now have the freedom to choose the anticommutative domain and define $cok_\ast$ and $accok_\ast$. However, there are still $cdef(T_\ast) = c - d$ elements left in $C$, that are not in the $cim(T_\ast)$ yet, thus the anticommutative domain has the cardinality of at most $N - c + d$. Moreover, at least $a$ elements must be in the $acdom(T_\ast)$, because $|acim(T_\ast)| = a$. For every possible cardinality $r$ of the $acdom(T_\ast)$, there are $\binom{N}{r}$ choices for the domain. For each of these choices, $accok_\ast$ must be a surjective mapping of $acdom(T_\ast)$ onto $A$, hence there are 
    \[
    \sum_{r = a}^{N-c+d}\binom{N}{r}\Surj(r, a)
    \]
    options for $accok_\ast$. We multiply each term of that sum by $2^r$ to account for all possible $ori_\ast$. The only thing left to define is the $ccok_\ast$. which must be a mapping from $\omega \setminus acdom_(T_\ast)$ to $C$. Note that it does not have to be surjective, because $D \subset C$ is already in the pairmorph image. However, $ccok_\ast$ must have a nonempty preimage for every element in $C \setminus D$. We utilize an argument built on the inclusion-exclusion principle to conclude that the total number of such mappings is
    \[
        c^{N-r} + \sum_{k = 1}^{c - d}\left((-1)^k\binom{c-d}{k}(c-k)^{N-r}\right).
    \]

    The claim follows by applying the product rule.
 \end{proof}

\begin{example}
    Consider the operation $\ast$ from Example~\ref{decomp_example}. Then $diag(T_\ast) = \{(1, 1), (2, 2)\}$, $cim(T_\ast) = \Delta(X)$ and $acim(T_\ast) = \{(0, 1)\}$. We count in how many ways one can define an operation $\circ$ with the same pairmorph image and diagonal. There are $\Surj(3, 2) = 6$ options for $ker_\circ$. Then, because $cdef(T_\ast) = 1$ and $arank(T_\ast) = 1$, the cardinality of the anticommutative domain of $\circ$ must be from $1$ up to $2$.
    \begin{enumerate}
        \item If $|acdom(T_\circ)| = 1$, then we have $\binom{3}{1} = 3$ possible anticommutative domains, each with only one way to define the $accok_\circ$, and $2$ ways to define $ori_\circ$, hence there are $6$ options for $accok_\circ$ and $ori_\circ$ in this case. Then, the $ccok_\circ$ must map $2$ elements to $\Delta(X)$, with the preimage of $(0, 0)$ necessarily being nonempty, since $(0, 0) \in (cim(T_\circ) \setminus diag(T_\circ))$.

        If the first element is mapped to $(0, 0)$, there are $3$ options for the destination of the second one. The same logic applies to the second one. However, we must subtract $1$, since we counted the mapping where both elements go to $(0, 0)$ twice. Henceforth, there are $5$ option for $ccok_\circ$ in total. Therefore, the cokernel can be defined in $6 \cdot 5 = 30$ ways.
        
        \item If $|acdom(T_\circ)| = 2$, then we have $\binom{3}{2} = 3$ possible anticommutative domains, each with only one possible $accok_\circ$, and with $2^2 = 4$ options for $ori_\circ$. Therefore, there are $12$ options for this pair of mappings. Then there is only one valid choice of the commutative part of the cokernel, because the only commuting pair must be sent to $(0, 0)$.
        
    \end{enumerate}
    Both cases give us $30 + 12 = 42$ possible cokernels. Therefore, $\circ$ can be defined in $6 \cdot 42 = 252$ ways.

    \begin{example}\label{card_L_class_circ_lz} Let $X$ be of arbitrary finite cardinality $n$, and let $\omega$ be a tournament on $X$. We compute the cardinality of the $\L$--class of the left zeroes operation $\circ_{lz}$. Since it is the identity of $\M$, any element $\ast \in \M$ from its left Green's class will have $\M\ast = \M$. Therefore, the left Green's class of $\circ_{lz}$ corresponds to the set of all left-invertible elements of $\M$. But left-invertible elements of the full transformation semigroup are injective, thus in our finite case, they are bijective. Because $\M$ embeds into $\T_{X \times X}$, and since the inverse of a bijective pairmorph transformation is also a pairmorph transformation, the $\L$--class of $\circ_{lz}$ is precisely the set of all units of $\M$, denoted in \cite{Lopez2022} by $U(\M)$.

    Applying our formula for $d = c = n$ and $a = N$ results in the cardinality of $U(\M)$ being equal to
    \[
        \Surj(n, n) \binom{N}{N}\Surj(N, N)2^N\left(n^{N - N} + 0\right) = n! \cdot 1 \cdot N! \cdot 2^N \cdot 1 = n!\binom{n}{2}!2^{\binom{n}{2}},
    \] which is precisely the result obtained in \cite[Theorem~3.6]{Lopez2022}.
        
    \end{example}
    
\end{example}
\section{The equivalence relation $\pi_\ast$}
So far the pairmorph transformation perspective has given us a complete description and enumeration of principal left ideals. Aiming to perform the same for principal right ideals, we transfer an important tool from transformation semigroup theory into our context. We reuse the existing notation \cite[page~47]{GanyushkinMazor2009}, and define an equivalence relation on $X \times X$ induced by an operation $\ast \in \M$.

\begin{definition}
    For a binary operation $\ast \in \M$, define the equivalence relation $\pi_\ast$ on $X \times X$ in the following way:
    \[
        (a, b) \pi_\ast (u, v) \iff T_\ast(a, b) = T_\ast(u, v)\iff \begin{cases}
            a \ast b = u \ast v,\\
            b \ast a = v \ast u.
        \end{cases}
    \]
\end{definition}

In other words, $\pi_\ast$ induces a partition of $X \times X$ into the preimages of $T_\ast$. This relation coincides with the $\pi_\alpha$ relation used in transformation semigroup theory.
However, our context restricts the structure of this relation. Since $T_\ast$ always ``respects'' the $^\leftarrow$ operation, we know that if $(a, b) \pi_\ast (u, v)$, then  it must always be the case that $(b,a) \pi_\ast (v, u)$ as well. It turns out that this simple property limits all equivalence classes of $\pi_\ast$ to just two types.

\begin{definition}
    Let $C \subset X \times X$ be an equivalence class of $\pi_\ast$. We call $C$ a
    \begin{itemize}
        \item \textit{commutative $\pi_\ast$--class} if for every $(a, b) \in C$, the pair $(b, a)$ also belongs to $C$;
        \item \textit{anticommutative $\pi_\ast$--class} if for every $(a, b) \in C$, the pair $(b, a)$ does not belong to~$C$.
    \end{itemize}
\end{definition}

\begin{proposition}\label{comm-acomm-class-crit}
    Let $\ast \in \M$ be an arbitrary binary operation. Then the $\pi_\ast$--class of a pair $(a, b) \in X \times X$ is commutative if and only if $\ast$ commutes at $(a, b)$, and it is anticommutative if and only if $\ast$ does not commute at $(a, b)$. 
\end{proposition}

\begin{proof}
    Let $C$ be an equivalence class of some $(a, b) \in X \times X$.
    \begin{enumerate}
        \item Let $\ast$ commute at $(a, b)$. Then we know that $T_\ast(a, b) = T_\ast(b,a)$ meaning that $(a,b) \pi_\ast(b ,a)$. Consider any $(u, v) \in C$. We also know that $(v, u) \pi_\ast (b, a)$, and since $\pi_\ast$ is transitive and $(b, a) \pi_\ast (a, b)$, we get that $(v, u)$ also belongs to $C$. Thus, $C$ is a commutative class.

        Now let $C$ be a commutative class of some $(a, b) \in X \times X$. Since $C$ is commutative, we know that $(a,b) \pi_\ast (b, a)$. This means that $a \ast b = b \ast a$, so $\ast$ commutes at $(a, b)$. \\

        \item Let $\ast$ not commute at $(a, b)$. Then we know that $a \ast b \neq b \ast a$. Because of this, $(b, a) \notin C$. Now for any $(u, v) \in C$, we have $(v, u) \pi_\ast (b, a)$, and $(b, a) \notin C$. Because $\pi_\ast$ is an equivalence relation, $(v, u)$ does not belong to $C$. Thus, $C$ is an anticommutative class.
        
        Now let $C$ be an anticommutative class of some $(a, b)$. Then $(b, a) \notin C$. From this it follows that $a \ast b \neq b \ast a$, hence $\ast$ does not commute at $(a, b)$.
    \end{enumerate}
        
\end{proof}

The proof of the second part of Proposition~\ref{comm-acomm-class-crit} hints at another crucial property of the relation $\pi_\ast$.

\begin{definition}
    For a $\pi_\ast$--class $C$, define its \textit{mirror class} as $C^\leftarrow := \{(b, a) : (a, b) \in C\}$. 
\end{definition}

\begin{proposition}\label{pi_ast_contains_all_mirror_classes}
    For any $\ast \in \M$, the relation $\pi_\ast$ contains all mirror classes. In other words, for every $\pi_\ast$--class $C$, its mirror class $C^\leftarrow$ is also a $\pi_\ast$--class.
\end{proposition}

\begin{proof}
    Let $C$ be the $\pi_\ast$--class of some $(a, b) \in X \times X$, and let $(x, y) = T_\ast(a, b)$. For any $(u, v) \in C$, we have $T_\ast(u, v) = (x, y)$. It follows that $T_\ast(v, u)= (y, x)$. Since the choice of $(u, v)$ was arbitrary, we conclude that $C^\leftarrow$ is a subset of some $\pi_\ast$--class $C'$. Now we show that actually $C' = C^\leftarrow$.\\
    
    Let $(v, u) \in C'$. Then $T_\ast(v, u) = (y, x)$, so $T_\ast(u, v) = (x, y)$. But then $(u, v) \in C$, thus $(v, u) \in C^\leftarrow$. Therefore, $C' = C^\leftarrow$, completing the proof.
\end{proof}

This property turns out to be the defining feature of $\pi_\ast$ relations. It encapsulates the limitation of classes to only commutative and anticommutative presented in Proposition~\ref{comm-acomm-class-crit}, as well as leads to a full characterization of all such relations.

\begin{theorem}\label{pi_ast_criterion}
    A given equivalence relation on $X \times X$ is the $\pi_\ast$ relation of some $\ast \in \M$ if and only if it contains all its mirror classes and there are no more than $|X|$ commutative classes.    
\end{theorem}
\begin{proof}
    Let $\ast \in \M$. Each commutative class of $\pi_\ast$ is precisely $T_{\ast}^{-1}(y, y)$ for some $y \in X$. Since there are no more than $|X|$ pairs $(y, y)$, the relation $\pi_\ast$ contains no more than $|X|$ commutative classes. The fact that $\pi_\ast$ contains all its mirror classes follows directly from Proposition \ref{pi_ast_contains_all_mirror_classes}. This shows the necessity of the condition. \\
 
    Let us now prove the sufficiency of this condition. Let $R$ be an equivalence relation on $X \times X$ such that our conditions are satisfied. We are going to build a binary operation $\ast \in \M$ by defining its kernel and cokernel such that $\pi_\ast$ will be equal to $R$.
  
    First, we partition $(X\times X)/R$ into three disjoint sets:
    \[
        (X \times X)/R = \mathcal{Y} \sqcup \mathcal{A} \sqcup \mathcal{B},
    \] 
    where $\mathcal{Y}$ is the set of all commutative classes of $R$, and $\mathcal{A}$ and $\mathcal{B}$ are such that $C \in \mathcal{A} \iff C^\leftarrow \in \mathcal{B}$. Note that this partitioning is not unique since the choice of $\mathcal{A}$ and $\mathcal{B}$ is not unique.

    To properly define every part of the commutative-anticommutative decomposition of the cokernel of $\ast$, we need to determine what $acdom(T_\ast)$ is going to be. Denote 
    \[
    M = \bigcup_{A \in \mathcal{A}}A; \; K = \bigcup_{Y \in \mathcal{Y}}Y,
    \]
    and set $acdom(T_\ast) = M \cup M^\leftarrow$. Note that $M^\leftarrow$ is precisely the union of all classes from $\mathcal{B}$. Let $\omega$ be the arc set of any tournament on $X$ that fully contains $M$. In other words, we assume that all elements of $M$ are in the ``default'' order. \\

    Because of our choice of $\omega$, the anticommutative part of the cokernel of $\ast$ is going to be a mapping from $M$ to $\omega$, and the commutative part will be a mapping from $K \setminus \Delta(X)$ to $\Delta(X)$. The kernel is going to be a mapping from $\Delta(X)$ to itself. The $ori_\ast$ can be defined as a constant zero. \\

    To make these mappings ``agree'' inside the same classes, we are interested in existence of two injections $\phi: \mathcal{Y} \to \Delta(X)$ and $\psi: \mathcal{A} \to \omega$. To show that such mappings exist, we analyze the cardinalities of $\mathcal{Y}$ and $\mathcal{A}$.

    The cardinality of $\mathcal{Y}$ is at most $|\Delta(X)| = |X|$ by our conditions. To reason about cardinalities of $\mathcal{A}$ and $\mathcal{B}$, note that the mapping $C \mapsto C^\leftarrow$ gives a bijection between $\mathcal{A}$ and $\mathcal{B}$, so $|\mathcal{A}| = |\mathcal{B|}$.

    The biggest in terms of the number of classes equivalence relation on $X \times X$ is the diagonal $\Delta(X\times X)$, which has $2\binom{|X|}{2}$ anticommutative classes, because every single ordered pair of distinct elements forms its own class. Thus, $|\mathcal{A} \cup \mathcal{B}| \le 2\binom{|X|}{2}$. Knowing that, and having $|\mathcal{A}| = |\mathcal{B}| = \frac12|\mathcal{A \cup B}|$, we can conclude that 
    \[
        |\mathcal{A}| = \frac12|\mathcal{A \cup B|} \leq \frac122\binom{|X|}{2} = \binom{|X|}{2} = |\omega|.
    \]

    Now we are sure that $\mathcal{|Y|} \leq |\Delta(X)|$ and $|\mathcal{A}| \leq |\omega|$, so finding injections $\phi$ and $\psi$, that we described above, is possible. With that in mind, we define every part of the commutative-anticommutative decomposition of $\ast$ by setting
    \[
    ker_\ast (a, a) = \phi(R(a, a)); \; ccok_\ast(a, b) = \phi(R(a, b)); \; accok_\ast(a, b) = \psi(R(a, b)); \; ori_\ast(a, b) = 0,
    \] where $R(a, b)$ denotes the $R$-equivalence class of $(a, b)$.

    Observe that by construction of $\ast$ for any $(a, b) \in im(T_\ast)$, we have $T_{\ast}^{-1}(T_\ast(a, b)) = R(a,b)$. Therefore, $\pi_\ast$ coincides with $R$.
   
 \end{proof}

\begin{example}
    Consider the set $X = \{0, 1, 2\}$ and an equivalence relation on $X\times X$ with the following classes:
    \[
    C_0 = \{(0, 0), (1, 1), (2, 2)\}; \; C_1 = \{(1, 0), (2, 1)\}; \; C_1^\leftarrow = \{(0, 1), (1, 2)\};
    \]
    \[
    C_2 = \{(2, 0)\}; \; C_2^\leftarrow = \{(0, 2)\}.
    \]

    This partition contains all mirror classes and has only one commutative class. Choosing $\mathcal{A} = \{C_1, C_2\}$ and setting $\phi(C_0) = (0, 0), \psi(C_1) = (1, -1), \psi(C_2) = (2, -2)$ produces the subtraction operation on $\mathbb{Z}_3$.
\end{example}

\section{Principal right ideals of $\M$}
With our new tool being set up and fully characterized, we are ready to describe arbitrary principal right ideals. Recal that similarly to principal left ideals, in a monoid $S$, the principal right ideal generated by an element $a \in S$ is the set $aS$ of all right products, and the right Green's relation is defined as $a \R b \iff aS = bS$. \\

An attentive reader may notice that the next characterization is identical to the one in transformation semigroups. However, we decided to prove it from scratch, instead of taking the characterization from $\T_{X \times X}$ and proving that it is appropriate in $\M$. We believe this provides a better understanding of the nature of right ideals in $\M$.

 \begin{theorem}
     For any $\ast \in \M$, the principal right ideal generated by it is precisely the set
     \[
        \ast\M = \{\dmd \in \M: \pi_\dmd \supset \pi_\ast\}.
     \]
 \end{theorem}
 \begin{proof}
    We first prove the inclusion from left to right. Let $\dmd \in \ast\M$ be an element of the principal right ideal generated by $\ast$. Then we know that $\dmd = \ast \trg \circ$ for some $\circ \in \M$.
    
    Suppose $(a, b) \pi_\ast (u, v)$ are two elements of $X \times X$ that are $\pi_\ast$--equivalent. Then $T_\ast(a, b) = T_\ast(u,v)$. From this we can conclude that 
    \[T_\dmd(a, b) = T_{\ast \trg \circ} (a,b) = T_\circ(T_\ast(a, b)) = T_\circ(T_\ast(u, v)) = T_{\ast \trg \circ}(u, v) = T_\dmd(u,v).\] 
    Therefore, $(a, b) \pi_\dmd (u, v)$. Thus, $\pi_\ast \subset \pi_\dmd$. \\
    
    Now going from right to left, let $\dmd \in \M$ be an arbitrary binary operation with $\pi_\ast \subset \pi_\dmd$. Our goal is to show that such $\dmd$ is always in the principal right ideal generated by $\ast$. For this, we need to construct an operation $\circ \in \M$ with $\dmd = \ast \trg \circ$.

    Consider the partition $(X \times X)/\pi_\ast$ of $X \times X$ into equivalence classes of $\pi_\ast$. Now take any choice function on $(X \times X)/\pi_\ast$ denoted by $c: (X \times X)/\pi_\ast \rightarrow X \times X$. Using this function, we define a mapping $\phi: im(T_\ast) \rightarrow X\times X$ by putting $\phi(x_1, x_2) = c(T^{-1}_\ast(x_1, x_2))$. Note that the preimage $T^{-1}_\ast(x_1, x_2)$ is always an element of $(X \times X)/\pi_\ast$. \\
    
    Consider an arbitrary pair $(y_1, y_2) \in im(T_\ast)$ and set $y_1 \circ y_2 = \dmd(\phi(y_1, y_2))$. For all pairs not in $im(T_\ast)$, define $\circ$ arbitrarily. Now, for any $a, b \in X$, we have 
    \[
    a (\ast \trg \circ) b = \circ(T_\ast(a, b)) = \dmd(\phi(T_\ast(a, b))).
    \] 

    To show that this is really equal to $a \dmd b$, we are going to use the fact that $\pi_\ast \subset \pi_\dmd$. Notice that 
    \[
        T_\ast(\phi(T_\ast(a, b))) = T_\ast(c(T^{-1}_\ast(T_\ast(a,b)))) = T_\ast(a, b).
    \]

    Therefore, $(a, b)$ and $\phi(T_\ast(a, b))$ are $\pi_\ast$--equivalent. Because of our conditions, $(a, b)$ and $\phi(T_\ast(a, b))$ are also $\pi_\dmd$--equivalent, thus $a \dmd b = \dmd(\phi(T_\ast(a, b))) = a (\ast \trg \circ) b$. This way, $\dmd = \ast \trg \circ$, completing the proof.

 \end{proof}

\begin{corollary}\label{R-equiv-criterion}
    Two operations $\ast, \star \in \M$ are $\R$--equivalent if and only if $\pi_\ast = \pi_\star$.
\end{corollary}

% In~\cite[Theorem~???]{Rafieipour2023}, the right Green's class of $\ast \in \M$ was shown to always coincide with the set of its right associates. Recall that \textit{a right associate} of an element $\ast \in M$ is an element $\star \in \M$ , such that there exists a unit $\dmd \in \M$ with $\ast = \star \trg \dmd$. Using this and Corollary~\ref{R-equiv-criterion}, we obtain the following.

% \begin{corollary}
%     A binary operation $\star \in \M$ is a right associate of $\ast \in \M$ if and only if $\pi_\ast = \pi_\star$.
% \end{corollary}

We follow the already established route and provide a full enumeration of all right Green's classes and their cardinalities.

\begin{theorem}
    There are exactly
    \[
        \sum_{c = 1}^n\sum_{k = 0}^N \binom{N}{k}S(n+k, c) \sum_{m = 1}^{N-k} S(N-k, m)2^{N-k-m}
    \] different principal right ideals in $\M$, where $N = \binom{n}{2}$, and $S(n, k)$ denotes the Stirling's number of the second kind.
\end{theorem}
\begin{proof}
    From the proof of Theorem~\ref{pi_ast_criterion}, we know that an arbitrary partitioning of $X \times X$ that contains all its mirror classes and has no more than $n$ commutative classes defines a unique $\R$--class. We approach this by counting only the partitions of $\Delta(X) \cup \omega$, and deciding what classes are going to be anticommutative.
    
    Let $1 \leq c \leq n$ be the number of commutative classes in the partitioning. These commutative classes consist of $n$ elements from $\Delta(X)$ together with some $k$ elements from $\omega$ with their mirrored versions, where $0 \leq k \leq N$. This $k$ elements can be selected in $\binom{N}{k}$ ways, and every such selection together with $\Delta(X)$ can be partitioned into $c$ classes in $S(n + k, c)$ ways. Henceforth, there are
    \[
        \sum_{c = 1}^n\sum_{k=0}^N\binom{N}{k}S(n+k, c)
    \]
    possible commutative classes.

   Next, we determine the anticommutative classes. There are $N - k$ remaining elements in $\omega$ after we fix the commutative classes. Denote those elements by $\omega'$. Our goal is to partition $\omega' \cup \omega'^\leftarrow$ such that the partitioning contains all mirror classes, and such that every class is anticommutative. This is equivalent to partitioning $\omega'$ into some classes, and then reversing the order of some elements we choose. For example, if $\omega' = \{(1, 2), (3, 4)\}$, we can partition $\omega'$ into $1$ class, and then reverse the order of $(1, 2)$ to obtain a relation with anticommutative classes $\{(2, 1), (3, 4)\}$ and $\{(1, 2), (4, 3)\}$.

    In general, let $\mathcal{P}$ be some partitioning of $\omega'$, and let $C \in \mathcal{P}$ be a class containing $l$ elements. Reversing the order of every element of a subset $A \subset C$ produces a new relation with classes $(C \setminus A) \cup A^\leftarrow$ and $(C^\leftarrow \setminus A^\leftarrow) \cup A$ instead of $C$ and $C^\leftarrow$. Notice that reversing the order of $C \setminus A$ produces the same relation. Therefore, for every choice of $A$, there is another choice producing the same relation, hence the number of possible relations obtainable by reversing the order of a subset $A \subset C$ is decreased by a factor of $2$, thus it is $\frac{2^l}{2} = 2^{l - 1}$, because there are $2^l$ subsets of $C$. \\

    Let now $\mathcal{P}$ be a partitioning of $\omega'$ into $m$ classes, each of cardinality $l_1, l_2, \dots, l_m$ respectively. For each class $C_i$, there are $2^{l_i - 1}$ relations obtainable by modifying it, henceforth there are 
    \[
    \prod_{i = 1}^m 2^{l_i-1} = 2^{\sum_{i = 1}^m(l_i - 1)} = 2^{(\sum_{i=1}^ml_i) - m} = 2^{|\omega'| - m} = 2^{N - k - m}
    \] impactful modifications of $\mathcal{P}$ in total. 

    From applying this logic to every possible partitioning of $\omega'$ into $m$ classes for every possible $m$, it follows that the number of partitions of $\omega' \cup \omega'^\leftarrow$ into mirror-preserved anticommutative classes is exactly 
    \[
    \sum_{m = 1}^{N-k} S(N-k, m)2^{N - k -m}.
    \]

    The claim follows by applying the product rule.

\end{proof}

\begin{proposition}
    Let $\ast \in \M$ be a binary operation with $crank(T_\ast) = c$ and $arank(T_\ast) = a$. Then the cardinality of its right Green's class is precisely
    \[
        2^a \frac{n!}{(n-c)!}\frac{N!}{(N-a)!}.
    \]
\end{proposition}
\begin{proof}
    Let $(X \times X)/\pi_\ast = \mathcal{Y \cup A \cup B}$ be the partitioning from the proof of Theorem~\ref{pi_ast_criterion}. From that proof we know that every binary operation $\star$ with $\pi_\star = \pi_\ast$ corresponds to some specific choice of $\mathcal{A}$ and a pair of injective mappings from $\mathcal{Y}$ to $\Delta(X)$ and from $\mathcal{A}$ to some $\omega$. From our conditions we also know that $|\mathcal{Y}| = c$ and $|\mathcal{A}| = a$.

    There are $\frac{n!}{(n-c)!}$ possible injections of $\mathcal{Y}$ to $\Delta(X)$, and there are $\frac{N!}{(N-a)!}$ injections of $\mathcal{A}$ to $\omega$. Now to account for all possible choices of $\mathcal{A}$, we multiply the total count by $2$ for every class in $\mathcal{A}$, because we can either keep the class as it is, or take its mirror class instead of it. The claim follows by applying the product rule.
\end{proof}

\begin{example}
    Let $X, \omega$ and $\ast$ be the same as in Example~\ref{decomp_example}. Then the relation $\pi_\ast$ contains $3$ commutative classes:
    \[
        C_0 = \{(0, 1), (1, 0)\}; \; C_1 = \{(1, 1), (2, 2)\}; \; C_2 = \{(0, 0)\}.
    \]
    Each of these classes must be mapped to some element of $\Delta(X)$, because otherwise the classes will not be commutative. The map must be injective because elements from different classes cannot have identical values of the pairmorph transformation. Hence, there are $3! = 6$ such mappings.

    Next, $\pi_\ast$ also contains an anticommutative class $C = \{(0, 2), (2, 1)\}$ and its mirror class. The value for $C$ can be chosen in $|\omega| = 3$ ways, and we can also choose to assign this value to $C^\leftarrow$ instead, hence there are 6 possible mappings. This gives us $6 \cdot 6 = 36$ options in total.
\end{example}

\begin{example}
    Let $X$ be of arbitrary finite cardinality $n$. Similarly to Example~\ref{card_L_class_circ_lz}, the right Green's class of $\circ_{lz}$ corresponds to the set of units $U(\M)$. Applying our formula with $c=n$ and $a = N$, again produces $|U(\M)| = n!N!2^N$, which as already mentioned, was established to be correct in \cite[Theorem~3.6]{Lopez2022}.
\end{example}

\section{Previous results on ideals within our framework}
With principal left and right ideals being fully described, we will go through some previously discovered ideals of $\M$, and show that our framework indeed accurately represents their behavior.

\begin{proposition}[\protect{\cite[Theorem~2.2]{Lopez2022}}]
    The principal right ideal of a constant binary operation is the set of all constant binary operations $\mathcal{C}$. The principal left ideal of an operation $\ast \in \M$ is a singleton if and only $\ast$ is constant. 
\end{proposition}
\begin{proof}
    If $c_d \in \mathcal{C}$ is a constant binary operation, then $im(T_{c_d}) = diag(T_{c_d}) = \{(d, d)\}$, and $\pi_{c_d} = \{X \times X\}$. There is no way to enlarge $\pi_{c_d}$ or to build a proper subset of $im(T_{c_d})$. The claim follows.
\end{proof}

\begin{proposition}[\protect{\cite[Proposition~9]{BinarySystems2008}}]
    For any $a \in X$, the set $Z_a=\{\circ \in \M\mid \A \; x \in X: x\circ x = a\}$ of all ``unique square'' operations is a left ideal in $\M$.
\end{proposition}

\begin{proof}
    For any $\ast \in Z_a$, we trivially have $diag(T_\ast) = \{(a, a)\}$, hence there is no way to reduce it. Therefore, all elements in its principal left ideal will also be in $Z_a$.
\end{proof}

\begin{proposition}[\protect{\cite[Proposition~13]{BinarySystems2008}}]
    The set $Ab(X)$ of all commutative operations is a two-sided ideal in $\M$.
\end{proposition}
\begin{proof}
    An operation $\ast \in \M$ belongs to $Ab(X)$ if and only if $im(T_\ast) \subset \Delta(X)$. Therefore, every element $\circ$ in the principal left ideal of $\ast$ will have $im(T_\circ) \subset im(T_\ast) \subset \Delta(X)$, so it will also belong to $Ab(X)$. Moreover, all $\pi_\ast$ classes of $\ast \in Ab(X)$ must be commutative, so all $\pi_\circ$ classes of any $\circ$ in the principal right ideal of $\ast$ will be commutative, since $\pi_\circ$ is an enlargement of $\pi_\ast$. Hence, $\circ$ must also belong to $Ab(X)$.
\end{proof}

\begin{proposition}[\protect{\cite[Proposition~14]{BinarySystems2008}}]
    The set $\overline{Str(X)}$ of all not anticommutative binary operations is a right ideal.
\end{proposition}
\begin{proof}
    An arbitrary $\ast \in \M$ is in $\overline{Str(X)}$ if and only if $\pi_\ast$ has a commutative class containing a pair not in $\Delta(X)$, hence every $\circ$ in the principal right ideal of $\ast$ has to also contain a commutative class with an element outside of $\Delta(X)$, because $\pi_\circ \supset \pi_\ast$. The claim follows.
\end{proof}

\section{Idempotents and regular elements}
Recall that an element $\ast \in \M$ is said to be \textit{idempotent} if $\ast^2 = \ast \trg \ast = \ast$.
Idempotent elements play a crucial role in understanding the structure of any semigroup. They are closely related to regular elements, maximal subgroups, and other fundamental structural components of semigroup theory.\\

In \cite{Lopez2022}, a small family of idempotent elements of $\M$ was described and some important observations were made. The set of all idempotent elements of $\M$ was concluded to be not closed under $\trg$, and a way to generate new idempotents from a given one was presented. Later in \cite{Rafieipour2023}, another family of idempotents (commutative ``tournament operations'') was described, and some combinatorial estimates were updated. \\

Now that we have a new framework in place, thus the description of arbitrary idempotents transfers directly from the transformation semigroup $\mathcal{T}_{X \times X}$.

\begin{proposition}\label{idempotent-crit}
    A given operation $\ast \in \M$ is an idempotent element if and only if it acts as $\circ_{lz}$ when restricted to $im(T_\ast)$. In other words, $T_\ast(a, b) = (a, b)$ for any $(a, b) \in im(T_\ast)$. 
\end{proposition}

\begin{proof}
    Follows directly from \cite[Theorem 2.7.2]{GanyushkinMazor2009}.
\end{proof}

\begin{example}
    Let $X = \{0, 1, 2\}$. Consider the operation $\ast \in \M$ given by the pairmorph graph depicted on Figure~\ref{fig:idemp}. Then $\ast$ is an idempotent element of $\M$, because $im(T_\ast) = \{(0, 0), (2, 2), (1, 2), (2, 1)\}\}$, and $\ast$ acts as $\circ_{lz}$ on the pairmorph image of $\ast$.
\begin{figure}[H]
    \centering
        \begin{tikzpicture}[node distance=16mm, scale=0.5,
          main/.style = {draw, circle, minimum size=0.6cm}]
            \node[main] (1) at (0,8) {$(0,0)$};
            \node[main] (2) at (0,4) {$(1,1)$};
            \node[main] (3) at (0,0) {$(2,2)$};
            \node[main] (4) at (3,7) {$(0,1)$};
            \node[main, fill=gray] (5) at (7,7) {$(0,2)$};
            \node[main] (6) at (5,2) {$(1,2)$};
        
            \draw[-{Latex[length=1.2mm,width=2.4mm]}]
              (1) edge[in = -110, out = -70, looseness=5] (1)
              (2) edge (3)
              (3) edge[in=-110, out=-70, looseness=5] (3)
              (4) edge[red, dashed] (6)
              (5) edge[red, dashed] (6)
              (6) edge[red, in=-110, out=-70, looseness=5, dashed] (6);
              
\end{tikzpicture}

    \caption{An idempotent element of $\M$ for $X = \{0, 1, 2\}$.}
    \label{fig:idemp}
\end{figure}
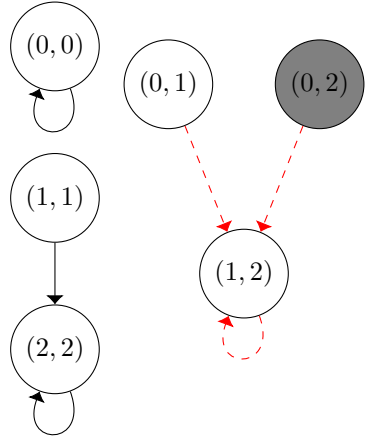

\end{example}

One immediately notices that all previously known families of idempotents act as left zeros operations on their pairmorph images. For example, constant binary operations $c_d$ have $im(T_{c_d}) = \{(d, d)\}$ and $T_{c_d}(d, d) = (d, d)$. On top of that, the pairmorph image of any ``tournament operation'' \cite[Definition~2.2.7]{Rafieipour2023} is a subset of $\Delta(X)$ because of commutativity, so by definition those operations always act as the left zeros operation on their pairmorph images. In fact, these two sets can be nicely generalized into a broader family, which was noticed in \cite{Rafieipour2023}.
\begin{proposition}
    An arbitrary commutative idempotent binary operation on $X$ is an idempotent element in $\M$.
\end{proposition}
\begin{proof}
    Let $\ast \in \M$ be commutative and idempotent. Because of commutativity, $im(T_\ast) \subset \Delta(X)$, and by definition of being an idempotent operation, for any $x \in X$, we have $x\ast x = x$, hence $T_\ast(x, x) = (x, x)$. Therefore, $\ast$ is an idempotent element of $\M$.
\end{proof}

\begin{example}
    If $X$ is an arbitrary lattice, then its infimum and supremum operations $\wedge$ and $\vee$ are idempotent elements of $\M(X)$.
\end{example}

\begin{proposition}\label{idempotent-image}
    Any idempotent element $\dmd \in \M$ has $cim(T_\dmd) = diag(T_\dmd)$.
\end{proposition}
\begin{proof}
    The inclusion from right to left is always true. For every $(y, y) \in cim(T_\dmd)$ we have $T_\dmd(y, y) = (y, y)$ because $\dmd$ is an idempotent element of $\M$. Then $cim(T_\dmd) \subset diag(T_\dmd)$, completing the proof.
\end{proof}

With a complete understanding of idempotency in $\M$, we are capable of fully enumerating all idempotents.

\begin{theorem}\label{idempcount}
	The number of idempotent elements in $\M$ is exactly
	\[
		\sum_{d=1}^n\binom{n}{d}d^{n-d}\sum_{a=0}^{N}\binom{N}{a}(2a+d)^{N-a}.
	\]
\end{theorem}
\begin{proof}
    Let $D \subset \Delta(X)$ and $A \subset \omega$ be some sets. Denote $|D| = d$ and $|A| = a$. We count how many idempotent elements $\dmd$ have $diag(T_\dmd) = D$ and $acim(T_\dmd) = A$. Since $\dmd$ is an idempotent element of $\M$, its values on $D$ and $A$ are already determined. Then the are $d^{n-d}$ options for the $ker_\dmd$, because any mapping of $\Delta(X) \setminus D$ to $D$ is going to produce an idempotent element. \\

    Now we need to count all options for $cok_\ast$. There are $N - a$ values left to determine. Each of them may be mapped arbitrarily into either $A$, $A^\leftarrow$ or $D$. Note that these mappings do not impact the pairmorph image, since $T_\dmd$ is defined as identity on every element from $A \cup A^\leftarrow \cup D$. Therefore, there are $(2a + d)^{N - a}$ options for $cok_\dmd$. \\

    The claim follows from applying this logic to every possible $A \subset \omega$ and nonempty $D \subset \Delta(X)$.
\end{proof}

\begin{remark}
    During the review process for the present paper, the anonymous referee has pointed out that idempotent elements of $\M$ were characterized and enumerated in a recent article \cite{Lopez2026} by S. R. L{\'o}pez-Permouth (which was published after we have submitted the paper). 
    The approach used there is more direct, and it does not utilize our transformation-focused framework. There, in \cite[Theorems~2.9, 3.1]{Lopez2026}, two characterizations of idempotent elements of $\M$ are provided, which are equivalent to Proposition~\ref{idempotent-crit}.
    As a result, in \cite[Theorem~3.6]{Lopez2026}, a full enumeration of all idempotent elements of $\M$ is given, which is equivalent to ours from Theorem~\ref{idempcount}. To demonstrate the equivalence of two enumerations, below we simplify the lengthy formula obtained in \cite[Theorem~3.6]{Lopez2026} (recall that $N=\binom{n}{2}$):
    \begin{align*}
       &\sum_{d=1}^{n} \binom{n}{d} d^{N + n - d}
        +
        \sum_{d=1}^{n}\sum_{k=1}^{N}\sum_{t=0}^{N-k}
        \binom{n}{d}
        \binom{N}{k}
        \binom{N-k}{t}
        d^{n-d+t}
        (2k)^{N-k-t} = \\
        &=\sum_{d=1}^{n} \binom{n}{d} d^{N + n - d}
        +
        \sum_{d=1}^{n}\sum_{k=1}^{N}
        \binom{n}{d}
        \binom{N}{k}
        \sum_{t=0}^{N-k}\binom{N-k}{t}
        d^{n-d+t}
        (2k)^{N-k-t} = \\
        &=\sum_{d=1}^{n} \binom{n}{d} d^{N + n - d}
        +
        \sum_{d=1}^{n}\sum_{k=1}^{N}
        \binom{n}{d}
        \binom{N}{k}
        d^{n-d}
        \sum_{t=0}^{N-k}\binom{N-k}{t}
        d^t
        (2k)^{N-k-t} = \\
        &= \sum_{d=1}^{n} \binom{n}{d} d^{N + n - d}
        +
        \sum_{d=1}^{n}\sum_{k=1}^{N}
        \binom{n}{d}
        \binom{N}{k}
        d^{n-d}
        (d+2k)^{N-k} = \\
        &=\sum_{d=1}^{n}\sum_{k=0}^{N}
        \binom{n}{d}
        \binom{N}{k}
        d^{n-d}
        (d+2k)^{N-k} = \sum_{d=1}^{n}\binom{n}{d}d^{n-d}\sum_{k=0}^{N}
        \binom{N}{k}
        (d+2k)^{N-k}.
    \end{align*}
\end{remark}

Not only does Proposition~\ref{idempotent-crit} allow us to determine the exact number of idempotent elements of $\M$, it provides us with a clean criterion for regularity. Proposition~\ref{idempotent-image} establishes a restriction on the commutative image for idempotent elements. This brings us to a known result in semigroup theory, stating that an element of a semigroup is regular if and only if its $\L$ (or $\R$) class contains an idempotent \cite[Proposition~4.7.1]{GanyushkinMazor2009}.

\begin{proposition}\label{cim=diag}
    A given $\ast \in \M$ is a regular element of $\M$ if and only if $cim(T_\ast) = diag(T_\ast)$.
\end{proposition}

\begin{proof}
    If an operation $\ast \in \M$ is regular, then its $\L$--class contains an idempotent $\dmd \in \M$ with the same pairmorph image and diagonal. From Proposition~\ref{idempotent-image} we conclude that $cim(T_\ast) = diag(T_\ast)$, because $\dmd$ is an idempotent element of $\M$. \\
    
    Let now $\ast \in \M$ be an arbitrary operation with our condition satisfied. Define an operation $\dmd \in \M$ by setting $T_\dmd(x, y) = (x, y)$ for any $(x, y) \in im(T_\ast)$, and choosing some $d \in diag(T_\ast)$, setting $T_\dmd(x, y) = (d, d)$ for everything else. By construction $\dmd$ is an idempotent element of $\M$. Note that if our condition was false, then there would be an element $(y, y) \in cim(T_\ast)$ that is not in $diag(T_\ast)$. Then we would get $diag(T_\dmd) \neq diag(T_\ast)$, because $y \dmd y = y$. However, under our assumption this is impossible, therefore pairmorph images and diagonals of $\ast$ and $\dmd$ coincide, hence $\dmd$ is an idempotent in the $\L$--class of $\ast$, thus $\ast$ is regular.
\end{proof}

\begin{example}
    Theorem~2.13 from \cite{Lopez2022} constructs an operation $\ast \in \M$ on $X = \{0, 1\}$ with $\mathrm{cim}(T_\ast)\neq \mathrm{diag}(T_\ast)$ given by the pairmorph graph depicted on Figure~\ref{fig:nonreg}. This serves as a proof of non-regularity of $\mathcal{M}$. Here $(1, 1)$ belongs to $cim(T_\ast)$ because it has an incoming arrow from $(0, 1)$. However, it has no incoming arrows neither from $(0, 0)$ nor from $(1, 1)$, thus $cim(T_\ast) \neq diag(T_\ast)$.
\end{example}
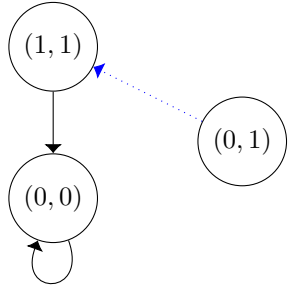
\begin{figure}[H]
    \centering
    \begin{tikzpicture}[node distance=16mm, scale=0.5,
  main/.style = {draw, circle, minimum size=0.6cm}]
    \node[main] (1) at (0,0) {$(0,0)$};
    \node[main] (2) at (0,4) {$(1,1)$};
    \node[main] (3) at (5,1.5) {$(0,1)$};

    \draw[-{Latex[length=1.2mm,width=2.4mm]}]
      (1) edge[in=-110, out=-70, looseness=5] (1)
      (2) edge (1)
      (3) edge[blue, dotted] (2);

    \end{tikzpicture}
    \caption{A non-regular element of $\M$ for $X = \{0, 1\}$.}
    \label{fig:nonreg}
\end{figure}
\begin{example}
    The multiplication operation $\cdot$ on $\mathbb{R}$ is not regular in $\M(\mathbb{R})$, because $(-1, -1) \in  im(T_\cdot)$ and is not in $diag(T_\cdot)$. On the other hand, the operation $+$ on $\mathbb{R}$ is regular in $\M(\mathbb{R})$.
\end{example}

\begin{theorem}
	The number of regular elements in $\M(X)$ is exactly
	\[
		\sum_{d=1}^n\binom{n}{d}\left(n^2 - n +d\right)^{N}\Surj(n, d).
	\]
\end{theorem}
\begin{proof}
    Let $D \subset \Delta(X)$ be a nonempty set. Denote $|D| = d$. We count all regular elements $\ast$ that have $D$ as its diagonal image. All possible kernels correspond to all possible surjections of $\Delta(X)$ onto $D$, hence there are $\Surj(n, d)$ of them. For $\ast$ to be regular, elements of $\omega$ can not be mapped to $\Delta(X) \setminus D$. Therefore, there are $(n^2 - n +d)^N$ possible cokernels. The claim follows by applying the product rule and repeating this argument for every choice of nonempty $D \subset \Delta(X)$. 
\end{proof}

As a closing result for this section, we are going to prove a conjecture from~\cite{Rafieipour2023}.
\begin{definition}[\protect{\cite[Definition~3.2.1]{Rafieipour2023}}]
    A binary operation $\ast \in \M(X)$ is called a \textit{two-value magma} if for any $x, y \in X$, we have $x \ast y \in \{x, y\}$.
\end{definition}
\begin{proposition}[\protect{\cite[Conjecture~6.2.7]{Rafieipour2023}}]
    The number of two-valued magmas that are idempotent elements of $\M$ is exactly
    \[
        \sum_{m = 0}^N\binom{N}{m}2^{N - m}=3^N.
    \]
\end{proposition}
\begin{proof}
    Let $\ast \in \M$ be a two-valued magma. By definition, for any two elements $x, y \in X$, the only options for the value of $T_\ast(x, y)$ are $\{(x, x), (x, y), (y, x), (y, y)\}$. Directly from this, $T_\ast|_{\Delta(X)}$ acts as a left zeros operation. Hence, the kernel of $\ast$ is already determined. \\

    Note that if for any two distinct $x, y \in X$ we have $T_\ast(x, y) = (y, x)$, then $\ast$ is not going to be an idempotent element of $\M$, since in that case $(y, x) \in im(T_\ast)$ and $T_\ast(y, x) = (x, y)$. On the other hand, any of the three other options for $T_\ast(x, y)$ keep $\ast$ acting as $\circ_{lz}$ on its pairmorph image. Therefore, for every element of $\omega$, there are three ways to define the cokernel. Hence, there are exatcly $3^N$ such cokernels, and, as a result, $3^N$ idempotent two-value magmas.
\end{proof}

\section{A note on the error concerning the center of $\M$}

This section is dedicated to addressing a problem in the 2011 paper of H.~F.~Fayoumi~\cite{Fayoumi2011Center}, regarding the center of the magma monoid. Recall that the \textit{center} of a semigroup $S$ is a subset $Z(S) \subset S$ of all elements that commute with every other element of $S$. In our case, the center of $\M$ is a subset $Z(\M) \subset \M$ consisting of all elements $\ast \in \M$ such that for any $\circ \in \M$, we have $\ast \trg \circ = \circ \trg \ast$. \\

The result in \cite[Theorem~3.1]{Fayoumi2011Center} states that the center of $\M$ corresponds to the set of all \textit{locally-zero} operations, i.e. the set of all operations $\ast \in \M$ such that for any $a, b \in X$, we have $T_\ast(a, b) \in \{(a, b), (b, a)\}$. However, this conclusion is incorrect.

As a counterexample, consider the set $X = \{0, 1, 2\}$, and let $\ast \in \M$ be given by the following pairmorph graph:
\begin{figure}[H]
    \centering
    \begin{tikzpicture}[node distance=16mm, scale=0.5,
  main/.style = {draw, circle, minimum size=0.6cm}]
    \node[main] (1) at (0,0) {$(1,1)$};
    \node[main] (2) at (0,4) {$(2,2)$};
    \node[main] (3) at (0,8) {$(0,0)$};
    \node[main,fill=gray] (4) at (5,4) {$(0,1)$};
    \node[main] (5) at (5,8) {$(0,2)$};
    \node[main] (6) at (5,0) {$(1,2)$};

    \draw[-{Latex[length=1.2mm,width=2.4mm]}]
      (1) edge[in=-110, out=-70, looseness=5] (1)
      (2) edge[in=-110, out=-70, looseness=5] (2)
      (3) edge[in=-110, out=-70, looseness=5] (3)
      (4) edge[red, in=-110, out=-70, looseness=5, dashed] (4)
      (5) edge[red, in=-110, out=-70, looseness=5, dashed] (5)
      (6) edge[red, in=-110, out=-70, looseness=5, dashed] (6);

    \end{tikzpicture}
    \caption{A locally-zero operation not in the center of $\M$.}
    \label{fig:bad}
\end{figure}
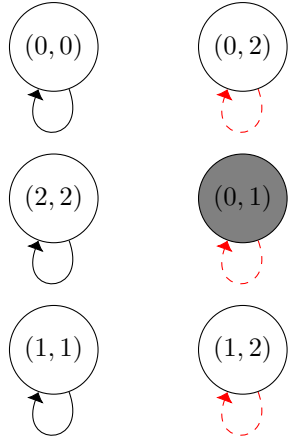
Consider any operation $\circ \in \M$ having $T_\circ(0, 1) = (1, 2)$. Then
\[
T_{\circ \trg \ast}(0, 1) = T_\ast(T_\circ(0, 1)) = T_\ast(1, 2) = (1, 2).
\]
At the same time, we know that
\[
    T_{\ast \trg \circ}(0, 1) = T_\circ(T_\ast(0, 1)) = T_\circ(1, 0) = (T_\circ(0, 1))^\leftarrow = (2, 1).
\]
Henceforth, $\circ \trg \ast$ is not equal to $\ast \trg \circ$, despite $\ast$ being a locally-zero operation. This counterexample arises from the fact that the proof of~\cite[Proposition 2.5]{Fayoumi2011Center}, which~\cite[Theorem 3.1]{Fayoumi2011Center} relies on, is incorrect. We close this gap by providing a correct characterization, demonstrating that the center of $\M$ is trivial.

\begin{proposition}\label{center_is_locally_zero}
    For any set $X$, the center of $\M(X)$ consists only of the left zeros and the right zeros operations. In other words,
    \[
        Z(\M) = \{\circ_{lz}, \circ_{rz}\}.
    \]
\end{proposition}

\begin{proof}
    These two operations clearly commute with every element of $\M$. We are interested in demonstrating that $\circ_{lz}$ and $\circ_{rz}$ are the only two such operations. Hence, let $\ast \in Z(\M)$ be an operation from the center of $\M$. 

    \textbf{Claim 1:} The operation $\ast$ is locally-zero.

    To the contrary, suppose that $\ast$ is not a locally-zero operation. Then there exists a pair $(a, b) \in X \times X$, such that $T_\ast(a, b) \notin \{(a, b), (b, a)\}$. Denote $T_\ast(a, b) = (x, y)$. Choose some $(u, v) \in X \times X$ such that $(u, v) \notin \{(a, b), (x, y)\}$, and consider an operation $\circ \in \M$ with $T_\circ(a, b) = (a, b)$ and $T_\circ(x, y) = (u, v)$. Note that if $x = y$, then $u$ must be equal to $v$. Observe that
    \[
        T_{\circ \trg \ast}(a, b) = T_\ast(T_\circ(a, b)) = T_\ast(a, b) = (x, y).
    \]
    On the other hand,
    \[
        T_{\ast \trg \circ}(a, b) = T_\circ(T_\ast(a, b)) = T_\circ(x, y) = (u, v).
    \]

    This means that $\circ \trg \ast \neq \ast \trg \circ$, which is a contradiction.

    \textbf{Claim 2:} $\ast \in \{\circ_{lz}, \circ_{rz}\}$.
    
    Suppose that $\ast \notin \{\circ_{lz}, \circ_{rz}\}$. Then there are distinct pairs of distinct elements $(a, b)$ and $(c, d)$, such that $\ast$ acts as the left-zeros operation on $(a, b)$ and as the right-zeros operation on $(c, d)$.  And as was shown above, any operation $\circ \in \M$ with $T_\circ(a, b) = (c, d)$ necessarily does not commute with $\ast$ under $\trg$. The obtained cotradiction finishes the proof.
\end{proof}

\section*{Acknowledgments} 
The authors are deeply grateful to the Armed Forces of Ukraine for keeping Kyiv safe, which made it possible for us to work on this paper. We also thank Serhii Slobodianiuk, Maryna Nesterenko, and Eugenia Kochubinska for a careful reading and insightful remarks that improved the presentation of the paper.

\bibliographystyle{plain}
\bibliography{references}% common bib file

\end{document}